\documentclass[11pt, reqno]{amsart}

\author[Paolo~Leonetti]{Paolo Leonetti}
\address[P.~Leonetti]{Department of Economics, Università degli Studi dell'Insubria, via Monte Generoso 71, Varese 21100, Italy}
\email{leonetti.paolo@gmail.com}
\urladdr{\url{https://sites.google.com/site/leonettipaolo/}} 

\author[Amir Khorrami Chokami]{Amir Khorrami Chokami}
\address[A. Khorrami Chokami]{ESOMAS Department, Universit\`a degli Studi di Torino, corso Unione Sovietica 128 bis, Turin 10134, Italy \\ and Collegio Carlo Alberto, Piazza Arbarello 8, Turin 10122, Italy}
\email{amir.khorramichokami@unito.it}


\title{The maximum domain of attraction of multivariate extreme value distributions is small}

\usepackage[T1]{fontenc}
\usepackage{amsmath}
\usepackage{amssymb}
\usepackage{amsthm}
\usepackage[left=3.5cm, right=3.5cm, paperheight=11.8in]{geometry}
\usepackage{hyperref}
\usepackage{fancyhdr}
\usepackage{enumitem}
\usepackage{stmaryrd}
\usepackage{comment}
\usepackage{nicefrac}
\usepackage{bm}
\usepackage{mathrsfs}
\usepackage{graphicx}
\usepackage[utf8]{inputenc}
\usepackage{cancel}
\usepackage{mathtools}

\newcommand{\vertiii}[1]{{\left\vert\kern-0.25ex\left\vert\kern-0.25ex\left\vert #1 
    \right\vert\kern-0.25ex\right\vert\kern-0.25ex\right\vert}}
		
\AtBeginDocument{%
   \def\MR#1{}
}

\newtheorem{thm}{Theorem}[section]
\newtheorem{cor}[thm]{Corollary}
\theoremstyle{definition} 
\newtheorem{defi}[thm]{Definition}
\newtheorem{rmk}[thm]{Remark}
\let\oldrmk\rmk
\renewcommand{\rmk}{\oldrmk\normalfont}

\newtheorem{claim}{\textsc{Claim}}
\newtheorem*{claim*}{\textsc{Claim}}

\hypersetup{
    pdftitle={{Almost all distributions belong to a domain of attraction}},
    pdfauthor={Paolo Leonetti and Amir Khorrami Chokami},
    pdfmenubar=false,
    pdffitwindow=true,
    pdfstartview=FitH,
    colorlinks=true,
    linkcolor=blue,
    citecolor=green,
    urlcolor=cyan
}

\providecommand{\MR}[1]{}

\providecommand{\MR}{\relax\ifhmode\unskip\space\fi MR }

\providecommand{\href}[2]{#2}

\newcounter{smallromans}

\newenvironment{romanenumerate}
{\begin{list}{{\normalfont\textrm{(\roman{smallromans})}}}%
  {\usecounter{smallromans}\setlength{\itemindent}{0cm}%
   \setlength{\leftmargin}{5.5ex}\setlength{\labelwidth}{5.5ex}%
   \setlength{\topsep}{.5ex}\setlength{\partopsep}{.5ex}%
   \setlength{\itemsep}{0.1ex}}}%
{\end{list}}

\usepackage{todonotes}

\newcommand{\Bigvi}{\scalebox{.8}[.8]{$\vee$}\hspace{-2.5mm}$\boxempty$}
\newcommand{\Vi}{{{\small \vee}\hspace{-2.2mm}\boxempty}}

\begin{document}

\maketitle
\thispagestyle{empty}

\begin{abstract}
    Consider the set of Borel probability measures on $\mathbf{R}^k$ and endow it with the topology of weak convergence. 
    We show that the subset of all probability measures which belong to the domain of attraction of some multivariate extreme value distributions is dense and of the first Baire category. 
    In addition, the analogue result holds in the context of free probability theory. 
\end{abstract}


\section{Introduction}\label{sec:intro}
\subsection{Informal premise} 
It is known that there exist probability laws, e.g., the Poisson law and the geometric law, which do not belong to the maximum domain of attraction $\mathsf{D}$ of any extreme value distribution (details will be given shortly), see e.g. \cite{MR256441, MR2969061}. 
At the same time, the structure of the set $\mathsf{D}$ has been extensively studied in the literature, see for instance \cite{
MR682807, MR841599, MR2364939, MR770637, MR1484316}. 
In addition, many results of extreme value theory rely on the hypothesis that the starting distributions belong to this maximum domain of attraction, see e.g. \cite{MR2963980, MR1483700, MR2757202, MR967580}. 

Essentially, here we are going to show that $\mathsf{D}$ is a topologically small set. 

\subsection{Main result} Fix an integer $k\ge 1$, and let $\mathscr{P}$ be the set of distributions on $\mathbf{R}^k$, which is endowed with the topology $\tau$ of weak convergence, so that the weak convergence of a sequence of distributions $(F_1,F_2,\ldots)$ in $\mathscr{P}$ 
to $F \in \mathscr{P}$ means that 
$$
\lim_{n\to \infty} \int_{\mathbf{R}^k} h\,\mathrm{d}F_n=\int_{\mathbf{R}^k} h\,\mathrm{d}F
$$
for every bounded and continuous $h: \mathbf{R}^k \to \mathbf{R}$, see e.g. 
\cite[Chapter 1]{MR1700749}. 

Let $(X_1,X_2,\ldots)$ be a sequence of i.i.d. random variables with values in $\mathbf{R}^k$. 
Let $F \in \mathscr{P}$ their common distribution, so that $F(x)=\mathrm{Pr}(X_1\le x)$ for all $x \in \mathbf{R}^k$, where $\le$ stands for the product pointwise order on $\mathbf{R}^k$. 
For each $n\ge 1$, define
$$
M_n:=\max\{X_1,\ldots,X_n\}.
$$

\begin{defi}\label{defi:main}
We say that $F$ belongs to the \emph{maximum domain of attraction}, denoted by $\mathsf{D}$, if there exist a sequence of functions $(u_1,u_2,\ldots)$ with $u_n: \mathbf{R}^k \to \mathbf{R}^k$ and a distribution $G\in \mathscr{P}$ such that: 
\begin{romanenumerate}
\item for each $n\ge 1$, $u_n$ is affine on each component, that is, there exist vectors $\alpha_n, \beta_n \in \mathbf{R}^k$ such that $\alpha_{n,1},\ldots,\alpha_{n,k}>0$ and 
$$
\forall x \in \mathbf{R}^k, \quad 
u_n(x_1,\ldots,x_k)=(\alpha_{n,1}x_1+\beta_{n,1},\ldots,\alpha_{n,k}x_k+\beta_{n,k});
$$
\item each marginal of $G$ is nondegenerate;
\item the sequence of distributions of $(u_1(M_1), u_2(M_2),\ldots)$ is weakly convergent to $G$. 
\end{romanenumerate}
\end{defi}

In such case, we say that $G$ is a \emph{multivariate extreme value distribution} and that $F$ belongs to the \emph{domain of attraction of $G$}, shortened as $F \in \mathsf{D}(G)$, see e.g. 
\cite[Section 5.4]{MR2364939}.
Denoting with $\mathscr{G}$ the family of multivariate extreme value distributions, we have 
$$
\mathsf{D}=\bigcup_{G \in \mathscr{G}}\mathsf{D}(G).
$$

The main question addressed here is to check the topological largeness of $\mathsf{D}$. 
It is well known that $\mathscr{P}$ is a Polish space (i.e., a separable and completely metrizable topological space). 
Hence, it follows by the Baire category theorem that $\mathscr{P}$ is not of the first Baire category in itself, see e.g. 
\cite[Theorem 2.5.2]{MR1932358}. 
Here, we recall that a subset $S\subseteq \mathscr{P}$ is said to be of the first Baire category if there exists a countable sequence $(S_1,S_2,\ldots)$ of $\tau$-closed sets with empty interior such that $S\subseteq \bigcup_{n\ge 1}S_n$; for instance, the set of $\mathbf{Q}$ of rational numbers is a set of the first Baire category in $\mathbf{R}$ (with its usual Euclidean topology), though it is topologically dense.  

In analogy with the latter example, our main result follows: 
\begin{thm}\label{thm:mainDsmallcategory}
The set $\mathsf{D}$ is dense and of the first Baire category in $\mathscr{P}$. 
\end{thm}

We provide an application for domain of attractions in the context of free probability theory, see e.g. \cite{MR2594350, MR2271490, MR1709310}. For, set $k=1$. Given $F,G \in \mathscr{P}$, define $H:=F$ \Bigvi \,$G \in \mathscr{P}$ by 
$$
\forall x \in \mathbf{R}, \quad 
H(x):=\max\{0,F(x)+G(x)-1\},
$$
see \cite[Definition 3.1]{MR2271490}. Let also $F^{\,\Vi n}$ be the $n$th fold $F$ \Bigvi\, $\cdots$ \,\Bigvi\,$F$. 
With these premises, a distribution $F \in\mathscr{P}$ is said to be in the \emph{maximum free domain of attraction} $\mathsf{D}^{\text{free}}$ if there exist a positive real sequence $(a_n)_{n\ge 1}$, a real sequence $(b_n)_{n\ge 1}$, and a distribution $G \in \mathscr{P}$ such that 
\begin{equation}\label{eq:freeconstraint}
\lim_{n\to \infty} F^{\,\Vi n}\left(a_nx+b_n\right)=G(x)
\end{equation}
for all continuity points $x \in \mathbf{R}$ of $G$. 
\begin{cor}\label{cor:freeprobability}
The set $\mathsf{D}^{\text{free}}$ is dense and of the first Baire category in $\mathscr{P}$. 
\end{cor}

Results in the same spirit, but completely different contexts, appeared, e.g., in \cite{AveniLeo22, MR4453356, Leo19, LeoTAUB}. 
In Section \ref{sec:preliminaries} we collect some preliminary results, while in Section \ref{sec:proof} we provide the proofs of Theorem \ref{thm:mainDsmallcategory} and Corollary \ref{cor:freeprobability}. 


\section{Preliminaries}\label{sec:preliminaries}

The key tool in the proof of our main result will be the following topological game:
\begin{thm}\label{thm:banachmazur}
Pick a subset $\mathscr{S}\subseteq \mathscr{P}$ and consider the Banach--Mazur game defined as follows\textup{:} 
Players I and II choose alternatively nonempty open subsets of $\mathscr{P}$ as a nonincreasing chain 
$$
U_1\supseteq V_1 \supseteq U_2 \supseteq V_2\supseteq \cdots, 
$$
where Player I chooses the sets $U_1,U_2,\ldots$; Player II has a winning strategy if 
$$
\bigcap\nolimits_{m\ge 1} V_m \cap \mathscr{S}=\emptyset.
$$ 

Then Player II has a winning strategy if and only if $\mathscr{S}$ is a set of the first Baire category in $\mathscr{P}$.
\end{thm}
\begin{proof}
It follows by \cite[Theorem 8.33]{MR1321597}. 
\end{proof}

In the following results, we are assuming that $k=1$ (they will be used explictly in the proof of Theorem \ref{thm:mainDsmallcategory}). 

\begin{thm}\label{thm:limitdistributions}
Set $k=1$. Then $\mathscr{G}=\{G_\gamma: \gamma \in \mathbf{R}\}$, where 
$$
\forall x \in \mathbf{R} \text{ with }1+\gamma x>0, \quad 
G_\gamma(x)=\mathrm{exp}\left(-(1+\gamma x)^{-1/\gamma}\right).
$$
In the case $\gamma=0$, it is interpreted as $G_0(x)=\mathrm{exp}(-e^{-x})$. 
\end{thm}
\begin{proof}
See \cite[Theorem 3]{MR8655}, or \cite[Theorem 1.1.3]{MR2234156} for a textbook exposition. 
\end{proof}

Accordingly, let $\{\mathsf{D}_+, \mathsf{D}_-, \mathsf{D}_0\}$ be a partition of $\mathsf{D}$ so that 
$$
\mathsf{D}_+:=\bigcup_{\gamma>0}\mathsf{D}(G_\gamma), \quad 
\mathsf{D}_-:=\bigcup_{\gamma<0}\mathsf{D}(G_\gamma), \quad 
\text{ and }\quad 
\mathsf{D}_0:=\mathsf{D}(G_0). 
$$
Lastly, 
for each $F \in \mathscr{P}$, define $x_F:=\sup\{x\in \mathbf{R}: F(x)<1\}$ and set 
$$
\mathsf{M}:=\{F \in \mathscr{P}: x_F<\infty\}.
$$

\begin{thm}\label{thm:necessaryconditions}
Set $k=1$ and pick a distribution $F \in \mathscr{P}$. 
\begin{enumerate}[label={\rm (\roman*)}]
\item \label{item:1necessary} If $F \in \mathsf{D}_{+}$ then $F\notin \mathsf{M}$ and
$$
\lim_{n\to \infty}\frac{1-F(2^{n+1})}{1-F(2^{n})}
$$
exists and belongs to $(0,1)$\textup{;}
\item \label{item:2necessary} If $F \in \mathsf{D}_{-}$ then $F \in \mathsf{M}$\textup{;} 
\item 
\label{item:3necessary} If $F \in \mathsf{D}_0\setminus \mathsf{M}$ then 
\begin{equation}\label{eq:limitgumbel}
\lim_{t\to \infty}\frac{1-F(t+f(t))}{1-F(t)}=\frac{1}{e},
\end{equation}
where 
$$
\forall t\in \mathbf{R}, \quad f(t):=\frac{\int_t^{\infty}(1-F(z))\,\mathrm{d}z}{1-F(t)}.
$$
In particular, $\int_0^\infty (1-F(z))\,\mathrm{d}z<\infty$. 
\end{enumerate}
\end{thm}
\begin{proof}
Point \ref{item:1necessary} follows by \cite[Theorem 1.2.1(1)]{MR2234156} setting $x=2$ along the subsequence $(2^n)$ of Equation (1.2.2). 
Point \ref{item:2necessary} is a consequence of  \cite[Theorem 1.2.1(2)]{MR2234156}. 
Finally, point \ref{item:3necessary} follows by setting $x=1$ in \cite[Theorem 1.2.5]{MR2234156}. Note that $1-F(t) \in (0,1)$ for all $t\ge 0$ since $F\notin \mathsf{M}$. 
The last "In particular" claim is obtained by the fact that, if $\int_0^\infty(1-F(z))\,\mathrm{d}z$ were not finite, then $f(t)=+\infty$ for all $t\ge 0$ and the limit in \eqref{eq:limitgumbel} would be equal to $0$. 
\end{proof}

\section{Main proofs}\label{sec:proof}

\begin{proof}
[Proof of Theorem \ref{thm:mainDsmallcategory}] 
Recall that $\mathscr{P}$ is a Polish space and, in addition, its weak topology $\tau$ is generated by the \emph{L\'{e}vy metric} $d$ defined by 
$$
d(F,G):=\inf\{\varepsilon>0: \forall x \in \mathbf{R}^k, F(x-\varepsilon e)-\varepsilon \le G(x) \le F(x+\varepsilon e)+\varepsilon\}
$$
for all $F,G \in \mathscr{P}$, where $e:=(1,\ldots,1) \in \mathbf{R}^k$, 
cf. \cite[Problem 8, p. 398]{MR1932358}.

Accordingly, for each distribution $F \in \mathscr{P}$ and $\varepsilon\in (0,1)$, let
$$
B(F,\varepsilon):=\{G \in \mathscr{P}: d(F,G)<\varepsilon\}
$$
be the open ball with center $F$ and radius $\varepsilon$, and define the set
\begin{displaymath}
\begin{split}
T(F,\varepsilon):=\{G \in \mathscr{P}: G(x)=F(x) \text{ for all }&x < F^{\leftarrow}(1-\varepsilon)e \\
&\text{ and }G(F^{\leftarrow}(1-\varepsilon)e)\ge 1-\varepsilon\},
\end{split}
\end{displaymath}
where $x<y$ means $x\le y$ and $x\neq y$, and 
$F^{\leftarrow}$ is defined by 
$$
\forall x \in \mathbf{R}\quad 
F^{\leftarrow}(x):=\inf\{t \in \mathbf{R}: F(te)\ge x\}
$$
(hence, $F^{\leftarrow}$ is the generalized inverse acting on the main diagonal of $\mathbf{R}^k$). 
Informally, $T(F,\varepsilon)$ is the family of distributions on $\mathbf{R}^k$ which coincide with $F$ for the initial $(1-\varepsilon)$ amount of probability in leftmost bottom part. 
\begin{claim}\label{claim:TincludedB}
Pick $0<\delta<\varepsilon<1$ and $F \in \mathscr{P}$. Then $T(F,\delta)\subseteq B(F,\varepsilon)$. 
\end{claim}
\begin{proof}
Fix $G\in T(F,\delta)$ and observe that 
\begin{displaymath}
\begin{split}
    F(x-\delta e)-\delta &\le 1-\delta\le G(F^{\leftarrow}(1-\delta)e)\le G(x)\\
    &\le \left(1-\delta\right)+\delta \le F(F^{\leftarrow}(1-\delta)e)+\delta\le F(x+\delta e)+\delta 
\end{split}
\end{displaymath}
for all vectors $x \ge F^{\leftarrow}(1-\delta)e$. This proves that $d(F,G) \le \delta$, therefore $G \in B(F,\varepsilon)$. 
\end{proof}

For each distribution $F \in \mathscr{P}$, we let $\mu_F$ be the unique probability measure defined on the Borel subsets $\mathscr{B}$ of $\mathbf{R}^k$ such that 
$$
\forall x \in \mathbf{R}^k, \quad 
F(x)=\mu_F\left(\,\prod\nolimits_{i=1}^k (-\infty,x_i]\right).
$$

\bigskip 

\textsc{Category of $\mathsf{D}$: case $k=1$.} First, let us prove that $\mathsf{D}$ is a set of the first Baire category for the one-dimensional case, hence let us assume hereafter that $k=1$. 

\begin{claim}\label{claim:Mnowheredense}
$\mathsf{M}$ is a set of the first Baire category in $\mathscr{P}$. 
\end{claim}
\begin{proof}
Since $\mathsf{M}=\bigcup_{t\ge 1}\mathsf{M}_t$, where $\mathsf{M}_t:=\{F \in \mathscr{P}: x_F \le t\}$, it is enough to show that each $\mathsf{M}_t$ is nowhere dense. 

For, fix $t\ge 1$ and observe that $\mathsf{M}_t$ is closed. 
In fact, 
pick $F \in \mathscr{P}\setminus \mathsf{M}_t$ so that $x_F>t$. Let $n$ be a positive integer such that $t+\frac{1}{n}<x_F$ and $F(t+\frac{1}{n})\le 1-\frac{1}{n}$: this integer exists because, in the opposite, by the right continuity of $F$ we would obtain 
$$
1>F(t)=\lim_{n\to \infty}F\left(t+\frac{1}{n}\right)\ge \lim_{n\to \infty}\left(1-\frac{1}{n}\right)=1,
$$
which is a contradiction. At this point, pick $G \in B(F,\frac{1}{2n})$. Since $d(F,G)\le \frac{1}{2n}$, it follows that 
$G\left(t+\frac{1}{2n}\right)\le F\left(t+\frac{1}{n}\right)+\frac{1}{2n}<
1-\frac{1}{2n}$, thus $x_G\ge t+\frac{1}{2n}$. 
This implies that the open ball $B(F,\frac{1}{2n})$ has empty intersection with $\mathsf{M}_t$. 

In addition, $\mathsf{M}_t$ has empty interior: indeed, for each $F \in \mathsf{M}_t$ and $\varepsilon>0$, the open ball $B(F,\varepsilon)$ contains a distribution $G$ with $x_G>t$. To this aim, pick $G \in T\left(F,\frac{\varepsilon}{2}\right)$ with the property that $G(F^{\leftarrow}(1-\frac{\varepsilon}{2}))$ is exactly $1-\frac{\varepsilon}{2}$, and $\mu_G(\{\lfloor t+1\rfloor\})=\frac{\varepsilon}{2}$. 
Therefore $G \in B(F,\varepsilon)$, thanks to Claim \ref{claim:TincludedB}, and $x_G=\lfloor t+1\rfloor>t$, so that $G\notin \mathsf{M}_t$.
\end{proof}

Putting together Theorem \ref{thm:limitdistributions}, Theorem \ref{thm:necessaryconditions}, and the above Claim \ref{claim:Mnowheredense}, it follows that it is sufficient to show that both sets 
$\mathsf{D}_+$ and $\mathsf{D}_0\setminus \mathsf{M}$ 
are of the first Baire category in $\mathscr{P}$. 

\begin{claim}\label{claim:D+meager}
$\mathsf{D}_+$ is a set of the first Baire category in $\mathscr{P}$. 
\end{claim}
\begin{proof}
Apply the Banach--Mazur game described in Theorem \ref{thm:banachmazur} to $\mathscr{S}=\mathsf{D}_+$. We construct recursively the strategy of Player II, together with a strictly increasing sequence $(n_m)_{m\ge 0}$ of nonnegative integers with $n_0:=0$. 

Fix an integer $m\ge 1$ and suppose that, at the $m$th move, Player I choose an open set $U_m$, and that the integers $n_0,n_1,\ldots,n_{m-1}$ have been already defined. 
Then $U_m$ contains an open ball $B(F_m,\varepsilon_m)$ for some $F_m \in \mathscr{P}$ and $\varepsilon_m \in (0,1)$. 
Let also $n_m$ be the smallest integer such that $n_m>n_{m-1}$ and  $2^{n_m}>1+F_m^{\leftarrow}(1-\frac{\varepsilon_m}{2})$. 
At this point, let $G_m \in T(F_m,\frac{\varepsilon_m}{2})$ be the unique distribution with the property that 
$$
G_m\left(F_m^{\leftarrow}\left(1-\frac{\varepsilon_m}{2}\right)\right)=1-\frac{\varepsilon_m}{2}
\quad \text{ and }\quad 
\mu_{G_m}(\{2^{n_{m}+1}+1\})=\frac{\varepsilon_m}{2}.
$$
It follows by Claim \ref{claim:TincludedB} that $G_m \in U_m$. Since $U_m$ is open there exists $\delta_m>0$ such that $B(G_m,\delta_m)\subseteq U_m$. In addition, we have by construction that 
$$
\mu_{G_m}((2^{n_m}-1,2^{n_m+1}+1))=0.
$$
Finally, the strategy of Player II is to choose the open set 
\begin{equation}\label{eq:choiceVm}
V_m:=B\left(G_m,\min\left\{\delta_m,\frac{\varepsilon_m}{2^{m}}\right\}\right). 
\end{equation}

To conclude the proof, let us prove that it is indeed a winning strategy, that is, the distribution $F \in \bigcap_m V_m$ does not belong to $\mathsf{D}_+$. Since $F \in V_m$ for all $m\ge 1$, it follows that $d(F,G_m)\le \varepsilon_m/2^{m}$, so that 
\begin{displaymath}
\begin{split}
    \left(1-\frac{\varepsilon_m}{2}\right)-\frac{\varepsilon_m}{2^{m}}&=G_m\left(2^{n_m}-\frac{\varepsilon_m}{2^{m}}\right)-\frac{\varepsilon_m}{2^{m}}\le F(2^{n_m})\\ 
    &\le F(2^{n_m+1})\le G_m\left(2^{n_m+1}+\frac{\varepsilon_m}{2^{m}}\right)+\frac{\varepsilon_m}{2^{m}}=\left(1-\frac{\varepsilon_m}{2}\right)+\frac{\varepsilon_m}{2^{m}}.
\end{split}
\end{displaymath}
This implies that 
$$
\forall m\ge 1, \quad 
\frac{1-F(2^{n_{m}+1})}{1-F(2^{n_{m}})}
\ge \frac{2^{-1}-2^{-m}}{2^{-1}+2^{-m}}
$$
therefore 
$$
\limsup_{n\to \infty}\frac{1-F(2^{n}+1)}{1-F(2^{n})}=1,
$$
so that $F\notin \mathsf{D}_+$ by Theorem \ref{thm:necessaryconditions}.\ref{item:1necessary}. 
\end{proof}

\begin{claim}\label{claim:D0menoMmeager}
$\mathsf{D}_0\setminus \mathsf{M}$ is a set of the first Baire category in $\mathscr{P}$. 
\end{claim}
\begin{proof}
It goes on the same lines of the proof of Claim \ref{claim:D+meager}, the only difference being the $n_m$ satisfies the additional condition $2^{n_m}>4m/\varepsilon_m$. The strategy of Player II is to choose the open set $V_m$ as in \eqref{eq:choiceVm}.

Similarly, to conclude the proof, let us prove that it is indeed a winning strategy. Pick the distribution $F \in \bigcap_m V_m$ and observe that 
\begin{displaymath}
\begin{split}
    \int_0^\infty (1-F(z))\,\mathrm{d}z 
    &\ge \int_{2^{n_m}}^{2^{n_m+1}} (1-F(z))\,\mathrm{d}z \\
    &\ge \int_{2^{n_m}}^{2^{n_m+1}} \left(1-G_m\left(z+\frac{\varepsilon_m}{2^m}\right)-\frac{\varepsilon_m}{2^m}\right)\,\mathrm{d}z \\
    &=2^{n_m}\cdot \left(1-\left(1-\frac{\varepsilon_m}{2}\right)-\frac{\varepsilon_m}{2^m}\right)\ge 2^{n_m}\cdot \frac{\varepsilon_m}{4} \ge m
\end{split}
\end{displaymath}
for all $m\ge 2$, which implies that $\int_0^\infty (1-F(z))\,\mathrm{d}z=\infty$. We conclude by Theorem \ref{thm:necessaryconditions}.\ref{item:3necessary} that $F \in \mathsf{D}_0\setminus \mathsf{M}$, completing the proof. 
\end{proof}

The conclusion (for the case $k=1$) follows putting together Claim \ref{claim:D+meager} and Claim \ref{claim:D0menoMmeager}. 

\bigskip 

\textsc{Category of $\mathsf{D}$: case $k\ge 2$.} Let us assume now that $k\ge 2$ and let $\mathsf{D}^{(1)}$ be the collection of marginal distributions $F^{(1)}$ on the first coordinate of all $F \in \mathsf{D}$. It is clear by Definition \ref{defi:main} that 
$$
\mathsf{D}\subseteq \{F \in \mathscr{P}: F^{(1)} \in \mathsf{D}^{(1)}\}.
$$
It follows by the previous case that there exist closed sets $(S^{(1)}_1,S^{(1)}_2,\ldots)$ of distributions over $\mathbf{R}$ such that $\mathrm{D}^{(1)}\subseteq \bigcup_n S^{(1)}_n$ and each $S^{(1)}_n$ has empty interior. For each $n\ge 1$, set $S_n:=\{F \in \mathscr{P}: F^{(1)} \in S^{(1)}_n\}$. Hence it is enough to show that each $S_n$ is nowhere dense in $\mathscr{P}$ (the proof is standard and we include it for the sake of completeness). 

For, suppose that $(F_1,F_2,\ldots)$ is weakly convergent to $F$, with $\{F_1,F_2,\ldots\} \subseteq S_n$. Then $(F_1^{(1)},F_2^{(1)},\ldots)$ is weakly convergent to $F^{(1)}$ by the open mapping theorem, and this implies that $F^{(1)} \in S_n^{(1)}$, hence $F \in S_n$. Therefore $S_n$ is closed. Lastly, let us assume for the sake of contradiction that $S_n$ contains a nonempty open set $U\subseteq \mathscr{P}$. Equivalently, for every sequence $(F_1,F_2,\ldots)$ in $\mathscr{P}$ weakly convergent to $F \in U$ we have $F_n \in U$ for all sufficiently large $n$. 
This implies, again by the open mapping theorem, that for every sequence $(F^{(1)}_1,F^{(1)}_2,\ldots)$ of distributions on $\mathbf{R}$ weakly convergent to $F^{(1)} \in V:=\{G^{(1)}: G \in U\}$ we have $F^{(1)}_n \in V$ for all sufficiently large $n$. Hence $V$ would be a nonempty open subset of $S^{(1)}_n$, which is impossible. This proves that each $S_n$ is nowhere dense, hence $\mathsf{D}$ is of the first Baire category in $\mathscr{P}$. 

\bigskip

\textsc{Denseness of $\mathsf{D}$.} Fix an open ball $B(F_0,\varepsilon)$, with $F_0 \in \mathscr{P}$ and $\varepsilon>0$. 
Let $F_\star \in \mathscr{P}$ be the distribution on $\mathbf{R}^k$  defined by $F_\star=\Phi^k$, where $\Phi$ is the standard normal distribution and define 
$$
x_0:=F_0^{\leftarrow}\left(1-\frac{\varepsilon}{2}\right)e
\quad \text{ and }\quad 
x_\star:=F_\star^{\leftarrow}\left(1-\frac{\varepsilon}{2}\right)e=\Phi^{-1}\left(\left(1-\frac{\varepsilon}{2}\right)^{1/k}\right)e. 
$$
At this point, let $F$ be the unique distribution on $\mathbf{R}^k$ such that 
\begin{equation}\label{eq:definitionFdenseness}
\forall x \in \mathbf{R}^k, \quad 
F(x):=
\begin{cases}
\,F_0(x)\,\,\,& \text{ if }x<x_0,\\
\,F_\star(x-x_0+x_\star)\,\,& \text{ otherwise}.\\
\end{cases}
\end{equation}
It follows by construction that $F \in T(F_0,\varepsilon/2)$, hence $F\in B(F_0,\varepsilon)$ by Claim \ref{claim:TincludedB}. To conclude the proof, it will be enough to show that $F \in \mathsf{D}$. 

First, it follows by \cite[Example 1.1.7]{MR2234156} that 
$$
\forall r \in \mathbf{R}, \quad 
\lim_{n\to \infty}(\Phi(a_nr+b_n))^n=G_0(r)=\mathrm{exp}(-e^{-r}), 
$$
with $a_n:=(2\log n-\log \log n-\log (4\pi))^{-1/2}$ and $b_n:=1/a_n$ for all integers $n\ge 2$. 
In particular, $\Phi \in \mathsf{D}_0\setminus \mathsf{M}$. 
Thanks to \cite[Theorem 3.2]{MR1700749} and the fact that $\mathbf{R}^k$ is separable, we obtain that $F_\star \in \mathrm{D}(G_\star)$, where $G_\star:=G_0^k \in \mathscr{G}$; more precisely, 
$$
\forall x \in \mathbf{R}^k, \quad 
\lim_{n\to \infty}\left(F_\star(a_nx_1+b_n,\ldots,a_nx_k+b_n)\right)^n=G_\star(x).
$$

We are going to show that $F \in \mathrm{D}(G_\star)$ as well. For, fix a vector $x=(x_1,\ldots,x_k) \in \mathbf{R}^k$ and note by the definition of $F$ in \eqref{eq:definitionFdenseness} that $F_\star(x)=F(x+x_0-x_\star)$ provided that $x\not<x_\star$. 
In particular, setting $c_n:=a_n$ and $d_n:=a_n(x_{0,1}-x_{\star,1})+b_n$ for all $n\ge 2$, it follows that $$
F_\star(a_nx_1+b_n,\ldots,a_nx_k+b_n)=
F(c_nx_1+d_n,\ldots,c_nx_k+d_n)
$$
whenever $(a_nx_1+b_n,\ldots,a_nx_k+b_n)\not< x_\star$. However, the latter inequality is certainly verified whenever $n$ is sufficiently large (since $a_n\to 0^+$ and $b_n\to \infty$ as $n\to \infty$). Therefore
$$
\lim_{n\to \infty}\left(F(c_nx_1+d_n,\ldots,c_nx_k+d_n)\right)^n=G_\star(x),
$$
and by the arbitrariness of $x$ we conclude that $F \in \mathsf{D}(G_\star)\subseteq \mathsf{D}$. 
\end{proof}

\begin{rmk}
As it is evident from the proof above, $\mathsf{D}(G_0^k)$ is dense in $\mathscr{P}$. 
\end{rmk}

\begin{proof}
[Proof of Corollary \ref{cor:freeprobability}] 
Thanks to \cite[Theorems 6.10]{MR2271490}, there are three possible type of limit distributions $G$ in \eqref{eq:freeconstraint} (which are different from the classical distributions $G_\gamma$ given in Theorem \ref{thm:limitdistributions}). 
At the same time, there is still a close relationship between their corresponding domain of attractions. Indeed \cite[Theorems 6.11--6.13]{MR2271490} imply that $\mathsf{D}^{\text{free}}=\mathsf{D}$. The conclusion follows by Theorem \ref{thm:mainDsmallcategory} (in the case $k=1$). 
\end{proof}


\section{Closing Remarks}

It follows by Theorem \ref{thm:mainDsmallcategory} that $\mathsf{D}$ is not closed. It is also easy to see that it is not open. For, fix an open ball $B(F,\varepsilon)$ with $F \in \mathsf{D}$ and $\varepsilon>0$. Then it is enough to pick a distribution $G \in T(F,\varepsilon/2)$ such that $\int_0^\infty (1-G(z))\,\mathrm{d}z=\infty$ and $\mu_G([2^n,2^{n+1}])=0$ for infinitely many $n$. Then $G\notin \mathsf{D}$, thanks to Theorem \ref{thm:necessaryconditions} and Claim \ref{claim:TincludedB}. We leave as open question for the interested reader to establish the topological complexity of $\mathsf{D}$ and, in particular, whether it is a Borel set. 

In addition, it would be interesting to have a measure [non-]analogue of Theorem \ref{thm:mainDsmallcategory}. 
Indeed, it is well known that there exists a dichotomy between measure and category, see e.g. Oxtoby's classical book \cite{MR584443}. 
For instance, the set of normal numbers in $[0,1]$ has full Lebesgue measure and, on the other hand, it is of the first Baire category. 
In our case, the difficulty relies in the lack of a natural candidate of \textquotedblleft uniform probability measure\textquotedblright\, on $\mathscr{P}$. 

\subsection{Acknowledgments.} 
P.~Leonetti is grateful to PRIN 2017 (grant 2017CY2NCA) for financial support.

\bibliographystyle{plain}

\end{document}